\documentclass[letterpaper, 10 pt, conference]{ieeeconf}  
\usepackage[utf8]{inputenc} 
\usepackage[T1]{fontenc}    
\usepackage{hyperref}       
\usepackage{url}            
\usepackage{booktabs}       
\usepackage{amsfonts}       
\usepackage{nicefrac}       
\usepackage{microtype}      
\usepackage{lipsum}
\usepackage{graphicx}       
\graphicspath{{media/}}     
\usepackage{algorithm}%
\usepackage{cite}
\usepackage{subfigure}
\usepackage{wrapfig}
\usepackage{xcolor}
\usepackage{amssymb}
\usepackage{amsmath}
\usepackage{tcolorbox}

\usepackage{amsthm}
\newtheorem{assumption}{Assumption}
\newtheorem{dfn}{Definition}

\newtheorem{thm}{Theorem}
\newtheorem{lem}{Lemma}
\newtheorem{prop}{Proposition}
\newtheorem{rem}{Remark}
\usepackage{algpseudocode}
\usepackage{float}
\usepackage{algobox}
\makeatother
\usepackage{ifthen}
\newboolean{showcomments}
\setboolean{showcomments}{true}
\usepackage{color}
\usepackage{todonotes}

\definecolor{bleudefrance}{rgb}{0.19, 0.55, 0.91}
\definecolor{ao(english)}{rgb}{0.0, 0.5, 0.0}

\newcommand{\addcite}[0]{\ifthenelse{\boolean{showcomments}}
{\textcolor{purple}{(add cite(s)) }}{}}%

\newcommand{\addref}[0]{\ifthenelse{\boolean{showcomments}}
{\textcolor{purple}{(add ref) }}{}}%

\newcommand{\enrique}[1]{  \ifthenelse{\boolean{showcomments}}
{\todo[inline,color=bleudefrance]{Enrique: #1}}{}}
\newcommand{\agustin}[1]{  \ifthenelse{\boolean{showcomments}}
{\todo[inline,color=green!80!black]{Agustin: #1}}{}}
\newcommand{\rene}[1]{  \ifthenelse{\boolean{showcomments}}
{\todo[inline,color=cyan]{Ren\'e: #1}}{}}
\newcommand{\emmargin}[1]{\ifthenelse{\boolean{showcomments}}{\marginpar{\color{bleudefrance}\tiny EM: #1}}{}}
\newcommand{\hancheng}[1]{  \ifthenelse{\boolean{showcomments}}
{\todo[inline,color=orange]{Hancheng: #1}}{}}
\newcommand{\ziqing}[1]{  \ifthenelse{\boolean{showcomments}}
{\todo[inline,color=red]{Ziqing: #1}}{}}
\newcommand{\salma}[1]{  \ifthenelse{\boolean{showcomments}}
{\todo[inline,color=yellow]{Salma: #1}}{}}
\newcommand{\zxmargin}[1]{\ifthenelse{\boolean{showcomments}}{\marginpar{\color{purple}\tiny ZX: #1}}{}}
\newcommand{\stmargin}[1]{\ifthenelse{\boolean{showcomments}}{\marginpar{\color{red}\tiny ST: #1}}{}}

\newcommand{\hl}[1]{\ifthenelse{\boolean{showcomments}}
{\textcolor{red}{#1}}{#1}}

\newboolean{showedits}
\setboolean{showedits}{false}
\usepackage[commandnameprefix=always,markup=underlined]{changes}
\definechangesauthor[color=bleudefrance]{EM}
\newcommand{\aem}[1]{
\ifthenelse{\boolean{showedits}}
{\added[id=EM]{#1}}
{\!#1\hspace{-4.75pt}}
}
\newcommand{\repem}[2]{
\ifthenelse{\boolean{showedits}}
{\replaced[id=EM]{#1}{#2}}
{\!#1\hspace{-4.75pt}}
}
\newcommand{\dem}[1]{
\ifthenelse{\boolean{showedits}}
{\deleted[id=EM]{#1}}
{}
}

\newif \iffinal
\finalfalse
\iffinal

\newcommand{\WJcomments}[1]{{}}
\newcommand{\oldText}[1]{}
\newcommand{\toAppendix}[1]{}
\else
\newcommand{\SJmodified}[1]{{\color{blue} #1}}
\newcommand{\SJcomments}[1]{{\SJmodified{Shijie commented: #1}}}
\newcommand{\toAppendix}[1]{#1}
\newcommand{\oldText}[1]{#1}
\newcommand{\bx}{\mathbf{x}}
\newcommand{\bu}{\mathbf{u}}

\fi
\IEEEoverridecommandlockouts                              
\allowdisplaybreaks

\title{\LARGE \bf
Data-Driven MPC with Adaptively Sampled Non-Expert Demonstrations: Performance Guarantees and Sample Complexity
}

\author{
Shijie Pan$^*$,
Agustin Castellano and
Enrique Mallada%
\thanks{Shijie Pan,
Agustin Castellano and 
Enrique Mallada are with both Department of Electrical and Computer Engineering and Data Science and AI Institute, Johns Hopkins University, Baltimore, MD, USA. Emails: \{span34, acaste11, mallada\}@jhu.edu. $^*$Corresponding author: Shijie Pan}
\thanks{This work was supported by the NSF Global Centers program under Grant No.~2330450 and by the DOE Office of Science (ASCR) under Award No.~826565.}
}

\begin{document}

\maketitle

\thispagestyle{empty}
\pagestyle{empty}

\begin{abstract}
We study data-driven policy construction for dynamical systems using trajectories generated by surrogate finite-horizon Model Predictive Control (MPC), with the goal of approximating an infinite-horizon optimal control policy. We interpret these trajectories as non-expert demonstrations and propose a memory-based, nonparametric policy that constructs a Lipschitz-regularized upper envelope of the finite-horizon surrogate value function, which serves as an explicit upper bound on the value of the resulting policy, thereby providing a computable performance certificate while bypassing explicit optimization at runtime. We establish relative optimality guarantees with respect to the infinite-horizon optimal value function and derive sufficient MPC horizon conditions, together with sample complexity bounds, for achieving a prescribed relative error. Numerical experiments on a rocket landing task validate the theoretical predictions and demonstrate near-optimal performance.

\end{abstract}

\section{INTRODUCTION}
Model Predictive Control (MPC) has become an important tool for sequential decision-making problems, with applications in robotics \cite{mayne2014model}, autonomous racing \cite{rosolia2017autonomous}, and rocket landing \cite{jang2024convex}. MPC operates in a receding-horizon manner, where an infinite-horizon (or long horizon) optimal control problem is approximated by repeatedly solving a finite-horizon problem online. Achieving high performance typically requires a long prediction horizon to accurately approximate the infinite-horizon objective. However, in many real-time scenarios, control inputs must be computed under strict latency constraints, often at millisecond or even sub-millisecond time scales. This significantly increases online computational burden, often exceeding the real-time capabilities of CPUs, GPUs, or embedded processors \cite{jerez2014embedded}. As a result, MPC faces a fundamental time–accuracy tradeoff: long horizons improve performance but are computationally prohibitive, while shorter horizons are tractable but may lead to large optimality gaps \cite[Section~10.3]{boyd2004convex}. This motivates the development of efficient methods for computing near-optimal control policies in real~time.

A common {data-driven} approach is to learn {a} mapping $\pi(\bx)$ from the initial state $\bx$ to the optimal control $\bu^\star(\bx)$ via regression. 
Under specific structural assumptions, such as linear dynamics and quadratic stage costs, the MPC problem admits an analytical solution, leading to the framework of explicit MPC \cite{alessio2009survey}, where the optimal control policy is represented as a piecewise affine function of the state \cite{bemporad2002explicit}. 
For more general MPC problems where such analytical solutions are unavailable, { data-driven} approaches have been widely explored to {learn the} policy mapping. 
Deep neural networks (DNNs) are commonly used for this purpose, a paradigm often referred to as deep MPC \cite{chen2018approximating,cao2020deep,maddalena2020neural}. 
Beyond DNNs, other regression-based methods have also been investigated, including set-membership approximation \cite{canale2009set}, kernel regression \cite{huang2023robust}, Gaussian processes \cite{sasaki2018explicit}, and decision trees \cite{ren2025exact}. 
Although these approaches enable fast online implementation, they generally lack a clear characterization of the sample complexity required to achieve a prescribed performance guarantee, nor provide guarantees on monotonic improvement of policy performance {as more} data is collected.

As such, theoretical guarantees for data-driven MPC remain limited. For example, Hertneck et al. \cite{hertneck2018learning} provide Hoeffding-based probabilistic guarantees on the policy accuracy gap solely applicable to sampled trajectories under stability assumptions, while Tokmak et al. \cite{tokmak2025automatic} establish deterministic bounds on the policy gap and associated sample complexity guarantees under RKHS and geometric regularity assumptions on the policy mapping. However, these results rely on strong assumptions on the policy class.

In this work, we focus on two fundamental questions in (data-driven) MPC, related to the optimization horizon $N$ of the receding horizon controller.
\begin{enumerate}
    \item How does $N$ influence the performance of the receding horizon controller?
    \item If we use solutions from the receding horizon controller to build a data-driven policy, what can we say about its performance as a function of $N$?
\end{enumerate}

Many works have addressed the first question, both in terms of performance \cite{grune2009analysis, grune2008infinite} and stability \cite{braun2012predictive, grune2012nmpc, jadbabaie2005stability}, but usually requiring strong assumptions such as a Bellman inequality \cite{sontag1999control, grimm2005model}
. The second question is of great interest, and to the best of our knowledge hasn't been addressed before. In this paper, we establish lower bounds on $N$ that enable a particular family of policies---built with data from the receding horizon controller---to achieve a desired performance. Our work builds upon the recent work of Castellano et al.\cite{castellano2025data}, where the authors proposed a nonparametric data-driven policy that is greedy with respect to an upper bound of the optimal value function. Their method, however, relied on solutions to the {full-length} optimization problem---\emph{expert demonstrations}---while in this work we use solutions from the receding horizon controller---i.e. \emph{non-expert demonstrations}. In particular, we leverage data from the receding horizon controller (RHC) with horizon $N$ to construct a similar nonparametric policy with performance guarantees. Specifically, we make the following contributions:
\begin{enumerate}
    \item We derive lower bounds on the horizon length $N$ of the RHC to achieve a desired level of performance.
    \item We propose a nonparametric policy that uses trajectories from the RHC, and that is greedy with respect to a computable upper bound of its value function.
    \item Under mild assumptions, we establish a lower bound on $N$ ensuring this data-driven policy outperforms the  RHC's upper bound.
    \item We derive sample complexity results for our policy to achieve a desired performance.
\end{enumerate}

The rest of the paper is organized as follows. Section \ref{sec:prelim} establishes preliminaries on MPC, our overall goal, and lower bounds on the horizon of the receding horizon controller to get a prescribed performance. Section \ref{MPC} introduces our data-driven policy framework, and our two main results: conditions that make our policy \emph{improving} with respect to the RHC's upper bound, and the sample complexity needed to achieve a desired performance. Experiments in Section \ref{sec:experiments} empirically validate our framework.

\section{Preliminaries}\label{sec:prelim}
\subsection{Infinite-Horizon optimal control problem}
In this section, we introduce the system model and the optimal control formulation used throughout the paper. Consider a discrete-time invariant system,
\begin{align}
\bx_{t+1}=f(\bx_t,\bu_t),\,\bx_t\in\mathbb{X},\,\bu_t\in\mathbb{U},\,t \ge 0.
\label{sys}
\end{align}
where $\mathbb{X}\subset \mathbb{R}^{d_1}$ and $\mathbb{U}\subset \mathbb{R}^{d_2}$ are compact sets.
 The optimal control problem can be cast as the following 
optimization problem in $\left(\mathbb{R}^{d_2}\right)^{\infty}$. 
\begin{equation}
\begin{aligned}
J(\bx_0)=\min_{\{\bu_{t},\,t\geq0\}} \quad
& \sum_{t=0}^{\infty} \gamma^tl(\bx_t,\bu_t) \\
\text{s.t.}\quad
& \bx_{t+1} = f(\bx_t,\bu_t), \quad t \ge 0,\\
& \bu_t\in\mathbb{U},\qquad t \ge 0,\\
& \bx_0\in\mathbb{X} \text{ is given.}
\label{infhori}
\end{aligned}
\end{equation}
In \eqref{infhori}, $\gamma\in(0,1)$. Generally, infinite-horizon problems like \eqref{infhori} can be addressed in two ways: 
\begin{itemize}
 \item Directly finding the optimal value function $J$ via fixed point iteration (e.g. dynamic programming or RL). 
 \item Approximating \eqref{infhori} with a finite-horizon surrogate as in model predictive control (MPC).
\end{itemize}
In this work, we adopt the latter approach. The following assumptions are used throughout the paper.
\begin{assumption}[Lipschitz continuity]
System \eqref{sys} is Lipschitz continuous with respect to states and controls, that is $\forall \bx, \hat{\bx}\in\mathbb{X}$ and $\forall \bu, \hat{\bu}\in\mathbb{U}$ we have:
\begin{align*}
\|f(\bx,\bu)-f(\hat{\bx},\hat{\bu})\|_{\mathbb{X}}\leq L_{f}\|\bx-\hat{\bx}\|_{\mathbb{X}}+L_{u}\|\bu-\hat{\bu}\|_{\mathbb{U}},
\end{align*}
where $\|\cdot\|_{\mathbb{X}}$ and $\|\cdot\|_{\mathbb{U}}$ are norms defined in the state space $\mathbb{X}$ and action space $\mathbb{U}$, respectively. 
\label{a50}
\end{assumption}
\begin{assumption}[Non-negative cost]
In \eqref{infhori}, stage costs satisfy $l(\bx,\bu)\geq0,\,\forall \bx\in\mathbb{X},\,\bu\in\mathbb{U}$.\label{a501}
\end{assumption}
\begin{assumption}[Value function decay]\label{a5}
Optimal trajectories under \eqref{infhori} satisfy $J(\bx_0)\geq J(\bx_t),\,t\ge 0$.
\end{assumption}


\begin{assumption}[$\mathbb{X}$ regularity]
 $\mathbb{X}$ is a compact robust positively invariant set.
\label{a53}
\end{assumption}

We make the following remarks regarding these assumptions.
\begin{rem}
Assumption~\ref{a5} ensures that the value function decays sufficiently fast along optimal trajectories, despite discounting. This would be generally satisfied for tracking/stabilization problems where the convergence rate is faster than the discounting rate $\gamma$.
\end{rem}
\begin{rem}[Restriction on $\mathbb{X}$]
In this paper, we assume that $\mathbb{X}$ is a robust positively invariant set (Assumption~\ref{a53}) to focus on the optimality analysis within the limited page budget. By relaxing the requirement of expert demonstrations in \cite[Proposition~2]{castellano2025data}, this direction can be incorporated into our analysis as future work. Under this assumption, to simplify the notation, we impose the restriction only on the initial state in all MPC formulations \eqref{infhori}, \eqref{infhorib}, and \eqref{approx}.
\end{rem}
\subsection{Goal of the paper}
The goal of data-driven MPC is to learn a control policy $\pi$ from recorded state data and their corresponding  control inputs, over the state space $\mathbb{X}$. Given a policy $\pi:\mathbb{X}\to\mathbb{U}$, its associated value function is defined as:
\[
J_{\pi}(\bx) := \sum_{t=0}^{\infty} \gamma^t\, l(\bx_t, \pi(\bx_t)),
\]
where $\bx_{t+1} = f(\bx_t, \pi(\bx_t))$ with $\bx_0 = \bx$. 
The Bellman operator $\mathcal{T}^{\pi}$ associated with policy $\pi$ is given by
\[
\mathcal{T}^{\pi} \hat{J}(\bx) = l(\bx,\pi(\bx)) + \gamma \hat{J}\big(f(\bx,\pi(\bx))\big),\;\;\forall\hat{J}\in\mathcal{J},
\]
where $\mathcal{J} := \{ \hat{J} : \mathbb{X} \to \mathbb{R} \mid \hat{J} \text{ is bounded} \}$. The operator $\mathcal{T}^{\pi}$ is monotone and admits a unique fixed point \cite{bertsekas2011dynamic}. 
Further discussion is provided in Appendix I.

Our aim in this paper is for $J_\pi$ to satisfy the following relative error criterion with respect to the optimal value function $J(\bx)$:
\begin{align}
\sup_{\bx\in\mathbb{X}}\frac{J_{\pi}(\bx)-J(\bx)}{J(\bx)+\eta}\leq \mu,\label{MPCTag}
\end{align}
where $J_{\pi}(\bx)\geq J(\bx)$ always holds since $J(\bx)$ denotes the minimal value function of \eqref{infhorib} while $J_{\pi}(\bx)$ is the cost induced by the policy $\pi$. Without loss of generality, we assume tolerance
$\mu>0$. The constant {$\eta>0$} is introduced to ensure
the relative error is well-defined when $J(\bx)\to 0$.

Relative error is widely adopted in industrial optimization solvers (e.g., Gurobi \cite{gurobi_mipgap} and CPLEX \cite{cplex_epgap}). It is also an important performance metric in various application domains, including power systems \cite{sangrody2016initial} and autonomous racing \cite{guo2024time}. 
However, learning methods that explicitly optimize or provide guarantees in terms of relative error remain largely underexplored.
\subsection{Non-expert demonstrations}\label{s41}
To solve Problem \eqref{infhori}, it is customary to transform or approximate it by a finite horizon problem. We distinguish between two settings:
\begin{itemize}
    \item \textbf{Expert setting \cite[Assumption 2.5]{castellano2025nonparametric}:} the terminal cost is given by the exact value function $J(\bx_N)$:
    {\fontsize{9.5}{12}\selectfont{\begin{equation}
\begin{aligned}
J(\bx_0)=\min_{\{\bu_{0:N-1}\}} \quad
& \sum_{t=0}^{N-1} \gamma^tl(\bx_t,\bu_t) \;+\; \gamma^N J(\bx_{N}) \\
\text{s.t.}\quad
& \bx_{t+1} = f(\bx_t,\bu_t), \,t = 0,\ldots,N-1,\\
& \bu_t\in\mathbb{U},\,t = 0,\ldots,N-1,\\
& \bx_0\in\mathbb{X} \text{ is given.}
\label{infhorib}
\end{aligned}
\end{equation}}}
    \item \textbf{Non-expert setting:} the terminal cost $J(\bx_N)$ is approximated by a surrogate function $F(\bx_N)$:
    {\fontsize{9.5}{12}\selectfont{\begin{equation}
\begin{aligned}
J_N(\bx_0)=\min_{\{u_{0:N-1}\}} \quad
& \sum_{t=0}^{N-1} \gamma^tl(\bx_t,\bu_t) \;+\; \gamma^NF(\bx_{N}) \\
\text{s.t.}\quad
& \bx_{t+1} = f(\bx_t,\bu_t), \,t = 0,\ldots,N-1,\\
& \bu_t\in\mathbb{U},\,t = 0,\ldots,N-1,\\
& \bx_0\in\mathbb{X} \text{ is given.}
\label{approx}
\end{aligned}
\end{equation}}}
\end{itemize}

The expert setting requires solving the infinite-horizon problem to evaluate $J(\bx_N)$, which is generally impractical for data collection and control computation for a data-driven policy. 
Therefore, we focus on the non-expert setting. The challenge in this case is to find a suitable surrogate $F(\cdot)$.


Much attention has been given in the literature to the design of a terminal cost function $F(\cdot)$, particularly when stability is desired. For constrained LQR, it is customary to use a quadratic terminal cost $F(\bx_N) = \bx_N^\top P \bx_N$ where $P$ is the solution to the algebraic Riccati equation \cite{scokaert2002constrained}; this approach has been used for nonlinear systems too \cite{chen1998quasi}. The key idea, in general, is to design an $F(\cdot)$ that satisfies a one-step Bellman inequality \cite{jadbabaie2002unconstrained}.

In \eqref{approx}, we design a non-negative terminal cost $F(\bx_N)\geq 0$ that is always a lower approximate of $J(\bx_N)$ such that $F(\bx_N)\leq J(\bx_N),\,\forall \bx_N\in\mathbb{X}$. For any initial state $\bx_0 \in \mathbb{X}$, let $\bu_0^F(\bx_0)$ denote the first control input obtained by solving \eqref{approx}, $\bx_1^F(\bx_0)$ the corresponding predicted next state, and $\bx_N^F(\bx_0)$ the associated terminal state. Using this notation, we further make the following implicit assumptions on the terminal cost $F(\cdot)$.

\begin{assumption}[Stage-cost bounds]\label{a52}
Consider the approximate problem \eqref{approx}. For all $\bx_0\in\mathbb{X}$, there exists a constant $v>0$ such that
\begin{align}
l(\bx_0,\bu_0^F(\bx_0)) \ge v J_N(\bx_0).\label{scd}
\end{align}
Moreover, there exists a constant $C>0$ such that the terminal state $\bx_N^F(\bx_0)$ under \eqref{approx} satisfies
\begin{align}
J(\bx_N^F(\bx_0)) \le C\, J(\bx_0).\label{rbt}
\end{align}
\end{assumption}
Assumption \ref{a52} ensures that the stage cost at the initial state $l(\bx_0,\bu_0^F(\bx_0))$ is sufficiently large relative to the non-expert value function $J_N(\bx_0)$, 
while the true tail cost $J(\bx_N^F(\bx_0))$ in \eqref{approx} is bounded in terms of the initial expert value function $J(\bx_0)$. In the rocket landing experiment in Section \ref{sec:experiments} (nonlinear MPC with quadratic cost), we empirically show that the terminal cost can be chosen as $F \equiv 0$ for quadratic cost MPCs. Assumption~\ref{a52} holds when the stage cost $l$ and the non-expert objective $J_N$ exhibit the same order of growth. Further discussion is provided in Appendix~II.

The following lemma provides a simple but important observation.
Since the function $F(\cdot)$ is designed as a lower approximation of
the unknown tail cost $J(\cdot)$, replacing $J(\bx_N)$ with $F(\bx_N)$
in the truncated MPC problem results in a smaller (optimization) cost.
\begin{lem}[Value lower approximation]
Under Assumption \ref{a52}, $J_N(\bx)\leq J(\bx)$, $\forall\bx\in\mathbb{X}$.
\label{lem-2i}
\end{lem}
Lemma \ref{lem-2i} shows that the non-expert optimization cost $J_N$ always underestimates the optimal control cost $J$ in \eqref{infhorib}. The cost $J_N$, however, should not be confused with the value $J_{\pi_\mathrm{MPC}}$ associated with applying the MPC policy $\pi_{\mathrm{MPC}}$ via \eqref{approx}. Nevertheless, the closer $J_N$ is to $J$, e.g., as $N$ grows, the better the quality of the resulting control $\bu_0^F$. This  property will be used in the next section.

The next result characterizes how large the MPC horizon $N$ should be for the truncated problem to provide a good relative approximation.
\begin{prop}[Multiplicative suboptimality-based sufficient MPC horizon]\label{lem-1}
Under Assumptions~\ref{a501}-\ref{a52}, consider the non-expert problem~\eqref{approx}, then
\begin{equation}
(1-\gamma^N(1+C))J(\bx)
\le J_N(\bx),\,\forall\bx\in\mathbb{X}.
\label{eq:relative-bound}
\end{equation}
Moreover, if $N\geq \frac{\log\left({1+C}\right)}{\log\left(\frac{1}{\gamma}\right)}$, the following $\delta$-Bellman inequality holds:
\begin{equation}
\delta\, l(\bx_0,\bu_0^F(\bx_0))
\le
J_N(\bx_0)-\gamma J_N\!\bigl(f(\bx_0,\bu_0^F(\bx_0))\bigr),\, \forall \bx_0\in\mathbb{X},
\label{deltaclf}
\end{equation}
where,
\begin{equation}
\delta \le 1-\frac{C\gamma^N}{v(1-\gamma^N(1+C))}.
\label{eq:delta-bound}
\end{equation}
\end{prop}

The proof of Proposition \ref{lem-1} is shown in Appendix III.
\section{Data-driven Policy}
\label{MPC}
In this section we present our data-driven nonparametric policy, and provide explicit theoretical guarantees on the number of data points required to achieve $\mu$-relative optimality \eqref{MPCTag}.
\subsection{Non-expert nonparametric policy}
 Following the approach in \cite[Section 4]{castellano2025data}, we construct a nonparametric policy $\pi_{\mathcal{D}}$ by sampling initial states from the state space $\mathbb{X}$ and repeatedly solving \eqref{approx} to generate trajectories. 
In contrast to \cite{castellano2025data}, which assumes access to expert trajectories \eqref{infhorib}, our approach relies solely on non-expert trajectories generated by finite-horizon MPC \eqref{approx}, and thus does not require expert information. 

Let $k$ denote the trajectory index. For each sampled initial state, we generate a non-expert trajectory $\tau_k:=\{\tilde\bx_{t,k},\tilde\bu_{t,k}\}$ by repeatedly solving \eqref{approx}. The resulting tuples are stored in the dataset
\begin{align}
\mathcal{D} \triangleq \bigcup_k\{(\tilde\bx_{t,k},\tilde\bu_{t,k},\tilde{\mathbf{J}}_{t,k})\}_{t\ge0},
\label{dataset}
\end{align}
where $\tilde\bu_{t,k} = \bu_0^F(\tilde\bx_{t,k})$ and $\tilde{\mathbf{J}}_{t,k} = J_N(\tilde\bx_{t,k})$, with $\bu_0^F(\cdot)$ denoting the first optimal control input of \eqref{approx}.
\begin{dfn}[Non-expert nonparametric policy \cite{castellano2025data}]\label{def1}
Given a dataset $\mathcal{D}$ as in~\eqref{dataset} and a parameter
$\lambda > 0$, define,
\[
\pi_{\mathcal{D}} : \mathbb{X} \to \mathbb{U},\,
\pi_{\mathcal{D}}(\bx) = \tilde\bu_\ell(\bx),
\]
where $\ell(\bx) = \arg\min_{1\le i\le |\mathcal{D}|}
\left\{ \tilde{\mathbf{J}}_i + \lambda \|\bx-\tilde\bx_i\|_{\mathbb{X}} \right\}$.
If multiple minimizers exist, one of them is selected arbitrarily.
\end{dfn}
A suitable choice of $\lambda$ will be provided later in Theorem \ref{thm:pi-eval}. The proposed nonparametric policy $\pi_{\mathcal{D}}$ can be interpreted as a value-aware nearest neighbor based on the regularized score $\tilde{\mathbf{J}}_i+\lambda\|\bx-\tilde\bx_i\|_{\mathbb{X}}$, where the first control associated with the selected sample is applied to the system. The sampling mechanism used to collect this data will be described later on in Algorithm \ref{alg:1}.
\begin{figure*}[htb]\label{fig:flowchart}
    \centering
    \includegraphics[width=0.7\linewidth]{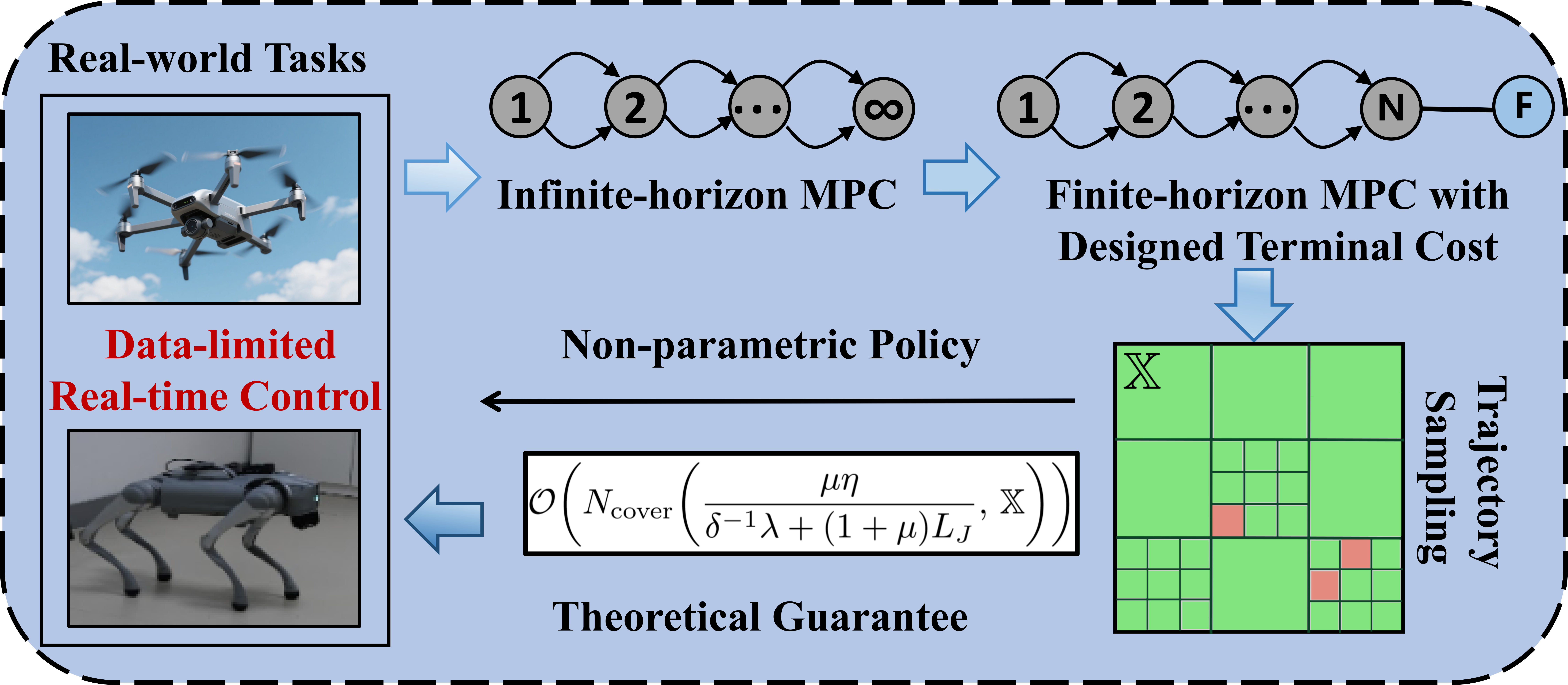}
    \caption{Non-expert data-driven MPC flowchart. An infinite horizon MPC problem is approximated via a finite horizon problem with a terminal cost function $F(\cdot)$. Trajectory solutions from this approximate problem are used to define our data-driven policy, which enjoys---with sufficient coverage---guarantees on performance.}
    \vspace{-10pt}
\end{figure*}
\subsection{Policy evaluation bound}
We now derive an upper bound for $J_{\pi_\mathcal{D}}(\bx)$ based on sampled non-expert values $\tilde{\mathbf{J}}_i$ contained in the dataset $\mathcal{D}$. First, we introduce the following assumption on the Lipschitz continuity of $J_N(\bx)$ and the stage cost $l(\bx,\bu)$ over state space $\mathbb{X}$.
\begin{assumption}[Lipschitz continuity]\label{a54}
For system \eqref{sys}, $J_N$ is Lipschitz continuous, i.e.,
\[
|J_N(\bx)-J_N(\hat{\bx})|\leq L_J\|\bx-\hat{\bx}\|_{\mathbb{X}},\,\forall\bx,\hat{\bx}\in\mathbb{X}.
\]
Similarly, the stage cost $l(\bx,\bu)$ is Lipschitz continuous in $\bx$, uniformly over $\bu \in \mathbb{U}$, i.e.,
\[
\sup_{\bu\in\mathbb{U}}|l(\bx,\bu)-l(\hat{\bx},\bu)|
\le
L_l\|\bx-\hat{\bx}\|_{\mathbb{X}},\,\forall\bx,\hat{\bx}\in\mathbb{X},
\]
and we assume there exists $\kappa>0$ such that $L_l\le \kappa L_J$.
\end{assumption}
From Assumption \ref{a54}, for any $\bx,\bx'\in\mathbb{X}$, we know,
$$
J_N(\bx) \leq J_N(\bx')  + L_J \|\bx-\bx'\|_{\mathbb{X}}\;.
$$
We then establish a global upper bound for $J_{\pi_\mathcal{D}}(\cdot)$, based on the data in $\mathcal{D}$.
\begin{dfn}\label{def2}
 For given $\lambda>0$ and $\delta\in(0,1)$, define:
    \begin{align*}
J_{\mathrm{ub}}^{\lambda}: \mathbb{X}\to\mathbb{R}:  J_{\mathrm{ub}}^{\lambda}(\bx) \triangleq \delta^{-1}\min_{1\leq i \leq |\mathcal{D}|}\left\{\tilde{\mathbf{J}}_i+ \lambda \|\bx-\tilde\bx_i\|_{\mathbb{X}}\right\},
    \end{align*}
where $\delta$ is the coefficient of Bellman inequality described in Proposition \ref{lem-1}.
\end{dfn}
\begin{assumption}[Dataset consistency]
\label{ass:data_consistency}
For any $(\tilde\bx_i, \tilde\bu_i, \tilde{\mathbf{J}}_i) \in \mathcal{D}$, there exists 
$(\tilde\bx_j, \tilde\bu_j, \tilde{\mathbf{J}}_j) \in \mathcal{D}$ such that
$\tilde\bx_j = f(\tilde\bx_i, \tilde\bu_i)$.
\end{assumption}
A preparatory lemma on functional monotonicity is presented below to facilitate the proof of Theorem~\ref{thm2}.
\begin{lem}[Policy value bound \cite{bertsekas2011dynamic}]
If $\hat{J} : \mathbb{X} \to \mathbb{R}$ satisfies
$(T^{\pi} \hat{J})(\bx) \le \hat{J}(\bx)$ for all $\bx \in \mathbb{X}$,
then $J_{\pi}(\bx) \le \hat{J}(\bx)$ for all $\bx \in \mathbb{X}$.\label{lem4}
\end{lem}
Our main result shows that, with appropriate choices of $\lambda$ and the MPC horizon $N$,
$J_{\pi_\mathcal{D}}(\bx) \leq J_{\text{ub}}^\lambda(\bx)$.
\begin{thm}[Policy evaluation inequality]\label{thm:pi-eval}
Under Assumptions~\ref{a50}-\ref{ass:data_consistency}, suppose that $\gamma L_f < 1$.
Given $1>\delta > 0$, if $N$ is greater than or equal to the following:
\begin{align}
{\fontsize{9}{12}\selectfont{ 
\max\left\{\!\!\tfrac{\log\left(\!-\!\frac{1}{4(1+C)}\!+\!\frac{1}{2}\sqrt{\frac{1}{4(1+C)^2}\!+\!\frac{2v(1\!-\!\delta)}{C(1+C)}}\right)}{\log\left({\gamma}\right)}\!,\!\tfrac{\log\left({2(1+C)}\right)}{\log\left(\frac{1}{\gamma}\right)}\!\!\right\}\!,\label{conN}}}
\end{align}
where both $C$ and $v$ are constants from Assumption \ref{a52}, then \eqref{deltaclf} holds. Furthermore, for any $\lambda$ satisfying $\lambda \ge 
\frac{ \kappa
}{
\delta^{-1}(1 - \gamma L_f)
}
\,L_J$, we have 
$J_{\pi_\mathcal{D}}(\bx) \le J_{\mathrm{ub}}^\lambda(\bx),\, \forall \bx \in \mathbb{X}$. \label{thm2}
\end{thm}
The MPC horizon condition \eqref{conN} provides a sufficient condition for the $\delta$-Bellman inequality \eqref{deltaclf} to hold. This inequality allows us to compare the value function before and after applying a given control input, which in turn yields a sufficient choice of $\lambda$.
\begin{proof}

\textbf{Step 1. Sufficient MPC horizon to ensure $\delta$-Bellman inequality of $J_N(\bx)$:} To ensure that the $\delta$-Bellman inequality of $J_N(\bx)$ in \eqref{deltaclf} holds, we seek the minimum horizon length $N$ such that \eqref{suffdelta} is satisfied, where \eqref{suffdelta} is obtained by rearranging \eqref{eq:delta-bound}.
\begin{align}
 \frac{C\gamma^N}{v(1-\gamma^N(1+C))} \le 1-\delta.\label{suffdelta}
\end{align}
We next derive a sufficient condition on the horizon length $N$ to satisfy \eqref{suffdelta}.
Since the denominator $1-\gamma^N(1+C)$ is positive when $\gamma^N(1+C) < 1$, 
we restrict to the regime where this condition holds. To further simplify the bound, we upper bound the LHS in \eqref{suffdelta} via, 
\begin{align*}
\frac{C\gamma^N}{1-\gamma^N(1+C)}&=C\gamma^N\left(1+\frac{\gamma^N(1+C)}{1-\gamma^N(1+C)}\right) \\&\le C\big(\gamma^N + 2(1+C)\gamma^{2N}\big),
\end{align*}
provided that $\gamma^N(1+C) \le \frac{1}{2}$. Thus, a sufficient condition for \eqref{suffdelta} is to solve the following quadratic inequality:
{\small{\begin{align}
&\gamma^N+2(1+C)\gamma^{2N}\leq\frac{v(1-\delta)}{C}\label{q0}\\
&\gamma^N\leq -\frac{1}{4(1+C)}+\frac{1}{2}\sqrt{\frac{1}{4(1+C)^2}+\frac{2v(1-\delta)}{C(1+C)}}\label{q1}.
\end{align}}}
Because $\frac{v(1-\delta)}{C}>0$, then the corresponding quadratic equation \eqref{q0} always has 2 roots. Taking logarithms on both sides (recalling that $\gamma \in (0,1)$) and combining with $\gamma^N(1+C) \le \frac{1}{2}$ yields the horizon bound \eqref{conN}. It is worth noting that for certain small values of $\delta$, the right-hand side of \eqref{q1} may exceed $1$. In this case, the constraint $\gamma^N(1+C)\leq \frac{1}{2}$ becomes effective. 

\medskip
\noindent
\textbf{Step 2. Bellman inequality of $J_{\mathrm{ub}}^\lambda$:} In the 2nd part of the proof, we show the $J_{\mathrm{ub}}^\lambda(\bx)$ satisfies the Bellman inequality described in Lemma \ref{lem4}, thus $J_{\pi_\mathcal{D}}\le J_{\mathrm{ub}}^\lambda$.
 From Definition \ref{def2}, we assign the upper bound as $J_{\mathrm{ub}}^\lambda(\bx_0)=\min\limits_{i\in[|\mathcal{D}|]}\delta^{-1}(\tilde{\mathbf{J}}_i+\lambda\|\bx_0-\tilde\bx_i\|_{\mathbb{X}})$. With a slight abuse of notation, we assume $i\in[|\mathcal{D}|]$  minimizes the right hand side of $J_{\mathrm{ub}}^\lambda(\bx_0)$ and denote $\tilde\bx_j=f(\tilde\bx_i,\tilde\bu_i)$. Then, { $\forall\bx_0\in \mathbb{X}$,}
\begin{align*}
&\mathcal{T}^{\pi_\mathcal{D}}J_{\mathrm{ub}}^\lambda(\bx_0)-J_{\mathrm{ub}}^\lambda(\bx_0)=l(\bx_0,\tilde\bu_i)+\gamma J_{\mathrm{ub}}^\lambda(\bx_0')-J_{\mathrm{ub}}^\lambda(\bx_0)\\
&\overset{(viii)}{\leq}l(\tilde\bx_i,\tilde\bu_i)+|l(\tilde\bx_i,\tilde\bu_i)\!-\!l(\bx_0,\tilde\bu_i)|+\gamma J_{\mathrm{ub}}^\lambda(\bx_0')\!-\!J_{\mathrm{ub}}^\lambda(\bx_0)\\
&\overset{(ix)}{\leq}\delta^{-1}(J_N(\tilde\bx_{i})-\gamma J_N(\tilde\bx_j))+\gamma J_{\mathrm{ub}}^\lambda(\bx_0')-J_{\mathrm{ub}}^\lambda(\bx_0)
\\&\quad +|l(\tilde\bx_i,\tilde\bu_i)-l(\bx_0,\tilde\bu_i)|\\&\overset{(x)}{\leq} \delta^{-1}(J_N(\tilde\bx_{i})-\gamma J_N(\tilde\bx_j))+\gamma J_{\mathrm{ub}}^\lambda(\bx_0')-J_{\mathrm{ub}}^\lambda(\bx_0)\\&\quad+
\kappa L_J\|\tilde\bx_i-\bx_0\|_{\mathbb{X}}
\\&\overset{(xi)}{\leq}\delta^{-1}\left[\gamma\lambda\|\bx_0'-\tilde\bx_j\|_{\mathbb{X}}-\lambda\|\bx_0-\tilde\bx_i\|_{\mathbb{X}}\right]+\kappa L_J\|\tilde\bx_i-\bx_0\|_{\mathbb{X}}\\
&\overset{(xii)}{\leq}(\delta^{-1}(\gamma\lambda L_f-\lambda)+\kappa L_J)\|\bx_0-\tilde\bx_i\|_{\mathbb{X}}\leq 0, 
\end{align*}
where $\bx_0'=f(\bx_0,\tilde\bu_i)$ is the next state obtained by applying $\tilde\bu_i$ on $\bx_0$, $\tilde\bu_i$ is the first optimal control by solving \eqref{approx} with initial condition $\tilde\bx_i$, and $i\in\arg\min_{k\in[|\mathcal{D}|]}\delta^{-1}(\tilde{\mathbf{J}}_k+\lambda\|\bx_0-\tilde\bx_k\|_{\mathbb{X}})$. The goal of step $(viii)$ is to replace the first term $l(\bx_0,\tilde\bu_i)$ by $l(\tilde\bx_i,\tilde\bu_i)$. Then, the $\delta$-Bellman inequality \eqref{deltaclf} is applied to obtain $(ix)$. Besides, step $(x)$ uses the Lipschitz condition (Assumptions \ref{a50} and \ref{a54}) to bound the comparison term on $l$.
Finally, $(xi)$ is obtained from the following expansion:
\begin{figure}[H]
\vspace{-7pt}
\centering
\begin{minipage}{0.55\linewidth}
\small
\begin{align*}
J_N(\tilde\bx_i) &= \tilde{\mathbf{J}}_i, \quad J_N(\tilde\bx_j) = \tilde{\mathbf{J}}_j,\\
J_{\mathrm{ub}}^\lambda(\bx_0') &\le \delta^{-1}(\tilde{\mathbf{J}}_j + \lambda \|\bx_0' - \tilde\bx_j\|_\mathbb{X}),\\
J_{\mathrm{ub}}^\lambda(\bx_0) &= \delta^{-1}(\tilde{\mathbf{J}}_i + \lambda \|\bx_0 - \tilde\bx_i\|_\mathbb{X}),
\end{align*}
\end{minipage}
\hfill
\begin{minipage}{0.35\linewidth}
\centering
\includegraphics[width=\linewidth]{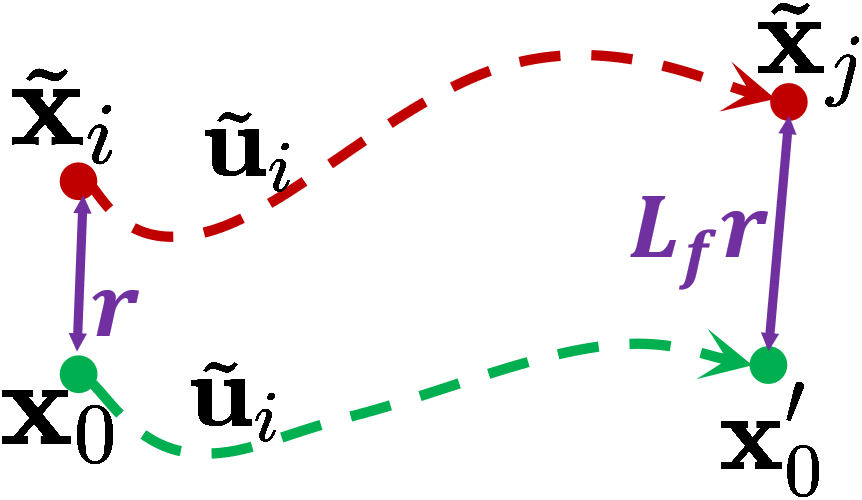}
\end{minipage}
\vspace{-6pt}
\end{figure}
\noindent where $\tilde{\mathbf{J}}_j=J_N(\tilde\bx_j)$ since \eqref{dataset}. Because of Assumption \ref{ass:data_consistency}, $\tilde\bx_j=f(\tilde\bx_i,\tilde\bu_i)$ is definitely in the data set $\mathcal{D}$. Substitute it into $(x)$ to get $(xi)$. Then, by further applying Assumption \ref{a50} (see Figure above, take $r=\|\bx_0-\tilde\bx_i\|_{\mathbb{X}}$), we have,
\begin{align*}
\delta^{-1}(\tilde{\mathbf{J}}_j+\lambda\|\bx_0'-\tilde\bx_j\|_{\mathbb{X}})&\leq \delta^{-1}(\tilde{\mathbf{J}}_j+\lambda L_f\|\bx_{0}-\tilde\bx_i\|_{\mathbb{X}}).
\end{align*}
By substituting it into $(xi)$, we obtain $(xii)$. 


Now, in order to apply Lemma \ref{lem4}, we need,
\begin{align*}
\delta^{-1}(\gamma\lambda L_f-\lambda)+\kappa L_J\leq 0\to
\lambda\geq  \frac{\kappa}{\delta^{-1}(1-\gamma L_f)}L_J,
\end{align*}
which also encodes the additional requirement $\gamma L_f<1$ to keep the denominator positive. 
Therefore, since $J_{\pi_\mathcal{D}}$ is the fixed point of $\mathcal{T}^{\pi_\mathcal{D}}$, we have $\mathcal{T}^{\pi_\mathcal{D}} J_{\pi_\mathcal{D}} = J_{\pi_\mathcal{D}}$. Moreover, $\mathcal{T}^{\pi_\mathcal{D}} J_{\mathrm{ub}}^\lambda \le J_{\mathrm{ub}}^\lambda$. By monotonicity of the Bellman operator (Lemma \ref{lem4}), it follows that $J_{\pi_\mathcal{D}} \le J_{\mathrm{ub}}^\lambda$. This completes the proof.
\end{proof}
Theorem \ref{thm:pi-eval} is closely related to Theorem 1 of \cite{castellano2025nonparametric} and Theorem 1 of \cite{castellano2025data}. In addition to the trajectory-level discrepancy considered there, we account for the \(\delta\)-expansion factor induced by the non-expert setting in \eqref{approx}, which avoids explicit use of the Bellman equation.
To compensate, we require a sufficiently large MPC horizon $N$ so that the $\delta$-Bellman inequality \eqref{deltaclf} holds, 
enabling policy evaluation via data-driven upper bounds. The condition $\gamma L_f < 1$ is a standard contraction requirement in the analysis of error propagation for discounted dynamic programming and Lipschitz systems; see, e.g., \cite{bucsoniu2018continuous,harder2024continuity,mahajan_stochastic_control_notes}.

\subsection{Adaptive Sampling with Performance Guarantees}
This section presents an active sampling mechanism for nonparametric data-driven policies  (Definition \ref{def1}) along with the theoretical guarantees.
We start by providing coverage conditions on the data that ensure a relative optimality gap of at most $\mu$. This will guide the design of our sampling algorithm.
\begin{thm}[Performance guarantees]\label{thm:performance-guarantees}
Let $\mu>0$, $\delta>\frac{1}{1+\mu}$, and $\eta>0$. Suppose that the policy $\pi_{\mathcal{D}}$ satisfies the conditions of Theorem~\ref{thm:pi-eval}, and that the dataset $\mathcal{D}$ satisfies the following coverage condition: for every $\bx\in\mathbb{X}$, there exists $\tilde\bx_i\in\mathcal{D}$ such that
\begin{align}\label{guarant}
\|\bx-\tilde\bx_i\|_{\mathbb{X}}
\le
\frac{(1-\delta^{-1})\tilde{\mathbf{J}}_i+\mu(\tilde{\mathbf{J}}_i+\eta)}
{\delta^{-1}\lambda+(1+\mu)L_J}.
\end{align}
Then the relative performance guarantee \eqref{MPCTag} holds. Moreover, the sample complexity is
\[
\mathcal{O}\!\left(
N_{\mathrm{cover}}\!\left(
\frac{\mu\eta}{\delta^{-1}\lambda+(1+\mu)L_J},
\,\mathbb{X}
\right)
\right),
\]
where $N_{\mathrm{cover}}(r,\mathbb{X})$ denotes the covering number of $\mathbb{X}$ with radius $r$.
\end{thm}

The proof is given in Appendix~IV. Theorem~\ref{thm:performance-guarantees} shows that the required coverage radius determines the achievable relative error: a covering of $\mathbb{X}$ at this resolution ensures \eqref{MPCTag}, and the corresponding sample complexity follows from standard covering arguments. The sample complexity is defined as the number of cells verified in Algorithm~1, which is equal to the number of memorized trajectories. Moreover, increasing $\delta$ (e.g., by increasing the MPC horizon used to generate the data) enlarges the admissible covering radius in~\eqref{guarant}, as illustrated in Table~1.  Importantly, the sample complexity is a worst-case sufficient bound and should not be interpreted as the number of trajectories required in practice. Tightening this bound is a significant direction for future work.

\begin{algorithm}[htb]
\caption{\bf {R}elative \bf{E}rror \bf{M}PC \bf{A}daptive \bf{S}ampling (REMAS) }\label{alg:1}
\begin{tcolorbox}[left=0mm,right=0mm,top=0mm,bottom=0mm]
{\fontsize{9.5}{12}\selectfont{\begin{algorithmic}
\State \textbf{Initialize:} $\mathcal{D}=\emptyset$; $\text{Unverified}=\mathbb{X}$; $\text{Verified}=\emptyset$.
\Repeat
    \For{$\mathbb{X}_k \in \text{Unverified}$}
        \State Get trajectory $\phi_k=\{(\tilde\bx_{t,k},\tilde\bu_{t,k},\tilde{\mathbf{J}}_{t,k})\}_{t=0}$.
        \State Trajectory from solving \eqref{approx} from $\tilde\bx_{0,k}$, center of $\mathbb{X}_k$.
        \State Append the trajectory $\phi_k$ to $\mathcal{D}$.
        \If{$(\tilde\bx_{0,k},\tilde{\mathbf{J}}_{0,k})$ perform over $\mathbb{X}_k$ (\eqref{guarant}, Thm \ref{thm:performance-guarantees})}
            \State $\text{Verified.add}(\mathbb{X}_k)$; $\text{Unverified.remove}(\mathbb{X}_k)$.
        \Else
            \State Split $\mathbb{X}_k$ into $3^n$ cells; add children to Unverified
        \EndIf
    \EndFor
\Until{$\text{Unverified}=\emptyset$}
\end{algorithmic}}}
\end{tcolorbox}
\end{algorithm}

To obtain a dataset $\mathcal{D}$ that satisfies Theorem~\ref{thm:performance-guarantees}, Algorithm~\ref{alg:1} adaptively partitions $\mathbb{X}$ to collect data for the nonparametric policy in Definition~\ref{def1}.
Each cell is represented by its center $\bx_i$, where a non-expert trajectory is generated by \eqref{dataset} and verified using Theorem~\ref{thm:performance-guarantees}. 
Cells that fail verification are recursively refined, while verified cells are retained. 
The resulting dataset ensures the policy satisfies \eqref{MPCTag}.

\section{Rocket landing experiments}\label{sec:experiments}
We consider a 2D planar rocket landing task modeled as a $6$D system. The state is given by $\bx = [p_x,\, p_z,\, v_x,\, v_z,\, \theta,\, \omega]^\top$,
where $p_x,p_z$ denote the horizontal and vertical positions, $v_x,v_z$ the corresponding velocities, $\theta$ the vehicle attitude, and $\omega$ the angular velocity. The control inputs are $\bu=[\tau_1,\tau_2]^\top$, where $\tau_1$ is the thrust magnitude and $\tau_2$ is the attitude torque. The control inputs are constrained as $0 \le \tau_1 \le 20,\,|\tau_2| \le 0.2$.
The continuous--time dynamics are,
\begin{align*}
\dot p_x \!=\! 0.3v_x,\,\dot p_z \!=\! 0.3v_z,\,
\dot v_x &\!=\! \frac{\tau_1}{3m}\sin\theta,\,\dot v_z \!=\! \frac{\tau_1}{3m}\cos\theta - \frac{g}{3}, \\
\dot \theta&\!=\! \omega,\,\dot \omega \!=\! \tau_2/I,
\end{align*}
where $m=1$ denotes the mass, $I=0.1$ the moment of inertia, and $g=9.8$ the gravitational acceleration.
The control objective is to land the rocket at the origin with zero velocity and upright attitude. The stage cost is,\[
l(\bx,\bu)=\bx^\top S \bx + \bu^\top R \bu,\ \bx\in\mathbb{X};\quad l=+\infty\ \text{otherwise},
\]
where $S, R \succ 0$, the discount factor is set to $\gamma = 0.8$, and the terminal cost is chosen as $F = 0$. This configuration empirically yields $C = 2.056$ and $v = 0.232$, which validate Assumption \ref{a52}. State constraints are imposed to represent feasible flight conditions. In particular, the state space $\mathbb{X}$ is defined as
\[
\mathbb{X}=\left\{
\begin{aligned}
|p_x|&\le 1, \quad 0 \le p_z \le 2, \quad |v_x|\le 1,\\
|v_z|&\le 1, \quad |\theta|\le 0.35, \quad |\omega|\le 1
\end{aligned}
\right\}.
\]

 We consider four different values of $N$ to study the relative error performance trend. Table~I reports the values of the $\delta$-Bellman inequality (computed from 
\eqref{eq:delta-bound},  Proposition~\ref{lem-1}) and the number of sampled trajectories (sample complexity) required by Algorithm~\ref{alg:1} for different choices of the horizon $N$, with $\eta = 3$ and $\mu = 1.2$. As indicated by Theorem~\ref{thm:performance-guarantees}, increasing the MPC horizon $N$ drives $\delta$ closer to $1$, which in turn reduces the number of trajectories required by Algorithm~\ref{alg:1} to satisfy~\eqref{guarant}.
\begin{table}[htb]\label{tab1}
\centering
\caption{$\delta$ and trajectory number on the MPC horizon $N$.}
\label{tab:delta_complexity}
\begin{tabular}{c c c c c}
\toprule
$N$ & 15 & 18 & 20 & 27 \\
\midrule
$\delta$ 
& $0.6504$ 
& $0.8309$ 
& $0.894$ 
& $0.9784$ \\
\midrule
Number of trajectories 
& $531441$ 
& $413505$ 
& $354537$ 
& $354537$ \\
\bottomrule
\end{tabular}
\end{table}

The large number of trajectories mainly results from the curse of dimensionality in covering the six-dimensional state space. We apply Algorithm~\ref{alg:1} across different horizon lengths $N$ on $50$ rocket landing tasks that uniformly sampled from $\mathbb{X}$.
\begin{figure}[htb]\label{fig:relative_error}
    \centering
    \includegraphics[width=1\linewidth]{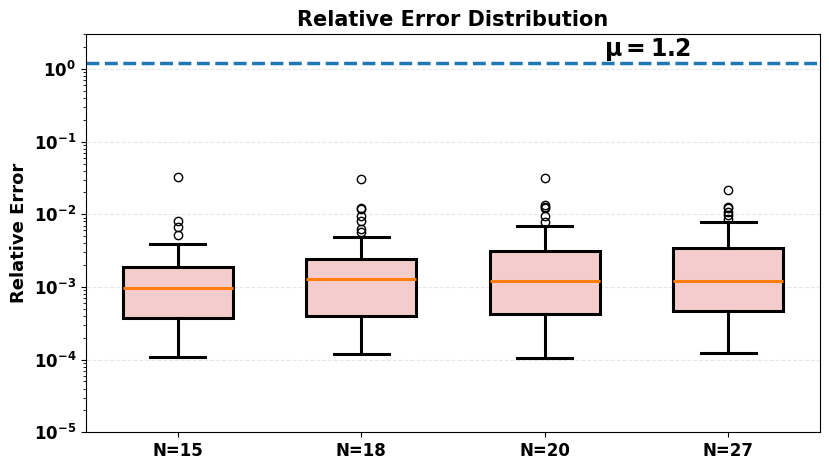}
    \vspace{-3ex}
    \caption{Distribution of the relative error ({\color{orange}{orange}} boxplots) in \eqref{MPCTag} for different choices of $N$ in 50 rocket landing tasks. Each distribution is much lower than the theoretical guarantee ({\color{blue}{blue}} dashed line) with $\mu=1.2$.}
\end{figure}

The relative error performance is shown in Fig.~2. The empirical relative error is lower than the theoretical guarantee claimed in \eqref{MPCTag} to a great extent. Although the mean and variance remain similar, the number of trajectories is significantly reduced as the MPC horizon 
$N$ increases. 
\section{Conclusion}\label{sec:conclusions}
In this paper, we proposed a data-driven nonparametric policy constructed from finite-horizon MPC trajectories without requiring expert demonstrations. 
We established $\mu$-relative performance guarantees with respect to the original infinite-horizon optimal value function and characterized sufficient conditions on the MPC horizon and data coverage to achieve this. 
Future work includes extensions to stochastic systems and improving scalability in high-dimensional settings.
\bibliographystyle{ieeetr}   
\bibliography{references}  
\appendices
\section{On the Bellman operator $\mathcal{T}^\pi$}\label{ap1}
The monotonicity and the fixed point uniqueness of $\mathcal{T}^\pi$ are defined as follows.\\
1. For any policy $\pi$, $\mathcal{T}^\pi$ is monotone, i.e.
\begin{align*}
&\hat{J}_1(\bx) \le \hat{J}_2(\bx),\,\forall \bx\in\mathbb{X},\,\forall \hat{J}_1,\hat{J}_2\in\mathcal{J}\\
\implies&
\bigl(\mathcal{T}^{\pi} \hat{J}_1\bigr)(\bx)
\le
\bigl(\mathcal{T}^{\pi} \hat{J}_2\bigr)(\bx),
\,\forall \bx\in\mathbb{X},\,\forall \hat{J}_1,\hat{J}_2\in\mathcal{J}.
\end{align*}
2. For any policy $\pi$, $J_{\pi}$ is the unique fixed point of $T^{\pi}$:
\[
\lim_{t \to \infty}
\underbrace{\bigl(
\mathcal{T}^{\pi} \circ \mathcal{T}^{\pi} \circ \cdots \circ \mathcal{T}^{\pi}
\bigr)}_{t\ \text{times}}
\, \hat{J}
= J_{\pi},\,\forall \hat{J} \in \mathcal{J}.
\]
where the uniqueness and existence of $J_\pi$ is guaranteed by $\gamma\in (0,1)$. 
\section{Interpretation of Assumption \ref{a52}}\label{ap2}
For the stage cost dominance condition \eqref{scd}, we argue that it is not restrictive. 
Based on \cite[Assumption 11, Remark 12]{rawlings2025nmpc_lecture}, if $0 \in \mathbb{X}$ is the unique equilibrium point, standard MPC assumptions imply the existence of $\alpha_1, \alpha_2 \in \mathcal{K}_{\infty}$ such that
\begin{align*}
l(\bx,u) &\geq \alpha_1(\|\bx\|_{\mathbb{X}}), \quad \forall \bx \in \mathbb{X},\; u \in \mathbb{U},\\
J_N(\bx) &\leq J(\bx) \leq \alpha_2(\|\bx\|_{\mathbb{X}}), \quad \forall \bx \in \mathbb{X}.
\end{align*}
If, in addition, $\alpha_1$ and $\alpha_2$ are of comparable growth, i.e., there exists $v>0$ such that $\alpha_1(r) \geq v\,\alpha_2(r),\, \forall r \geq 0$,
then it follows that
\[
l(\bx,u) \geq \alpha_1(\|\bx\|_{\mathbb{X}}) \geq v\,\alpha_2(\|\bx\|_{\mathbb{X}}) \geq v\,J_N(\bx),\forall \bx \in \mathbb{X},u \in \mathbb{U}
\]
which implies \eqref{scd}.
\section{Proof of Proposition \ref{lem-1}}\label{ap3}
We prove Proposition \ref{lem-1} by two parts. \\
\textbf{Part 1. Multiplicative bound between $J(\bx)$ and $J_N(\bx)$:} Without loss of generality, let $\bx=\bx_0$, for a given input sequence $\bu$, define the truncated cost
$V_N(\bx_0,\bu)=\sum_{t=0}^{N-1}\gamma^t l(\bx_t,\bu_t)$,
and the augmented cost
$V(\bx_0,\bu)=V_N(\bx_0,\bu)+\gamma^N J(\bx_N)$,
which corresponds to the finite-horizon cost with the true tail. Let $\bu^\star$ be the optimal control sequence associated with $J(\bx_0)$ in \eqref{infhorib}. To simplify notation, we denote $\bx_i^F(\bx_0)$ and $\bu_i^F(\bx_0)$ by $\bx_i^F$ and $\bu_i^F$, respectively, for $i=0,\ldots,N-1$. Then,
\begin{align}
&J(\bx_0)-J_N(\bx_0)
\\=& \underbrace{J(\bx_0)-V_N(\bx_0,\bu^\star)}_{(\text{I})}+ \underbrace{V_N(\bx_0,\bu^\star)-J_N(\bx_0)}_{(\text{II})}.\nonumber
\end{align}

Using the definition of $J(\bx_0)$ and Assumption~\ref{a5}, we bound (I) as,
\[
J(\bx_0)-V_N(\bx_0,\bu^\star)=\gamma^NJ(\bx_N)\le \gamma^N J(\bx_0).
\]
On the other hand, to bound (II), we let $\bu^F$ denote the optimal control sequence of $J_N(\bx_0)$. By optimality of $\bu^\star$ in \eqref{infhorib}, we have,
\[V(\bx_0,\bu^F)\ge V(\bx_0,\bu^\star)=J(\bx_0).\]
Then,
\[V_N(\bx_0,\bu^F)+\gamma^NJ(\bx_N^F)\overset{(i)}{{\ge}} V_N(\bx_0,\bu^\star)+\gamma^NJ(\bx_N),\]
  $(i)$ is obtained from expanding both sides from the definition of $V_N$. Next,
\begin{align*}
&V_N(\bx_0,\bu^F)+C\gamma^NJ(\bx_0)\overset{(ii)}\ge V_N(\bx_0,\bu^F)+\gamma^NJ(\bx_N^F)\\&\qquad\qquad\qquad\qquad\qquad\,\,\,\,\,\,\overset{(i)}\ge V_N(\bx_0,\bu^\star)+\gamma^NJ(\bx_N),
\end{align*}
 where applying \eqref{rbt} in Assumption~\ref{a52} leads to $(ii)$. At last, we have,
\begin{align*}
C\gamma^NJ(\bx_0)&\ge V_N(\bx_0,\bu^\star)-V_N(\bx_0,\bu^F)+\gamma^NJ(\bx_N)\\
&\overset{(iii)}{\ge} V_N(\bx_0,\bu^\star)-V_N(\bx_0,\bu^F)\\
&\geq V_N(\bx_0,\bu^\star)-J_N(\bx_0),
\end{align*}
where $(iii)$ is from rearranging $(ii)$ and applying $J(\bx_N)\ge 0$. Hence, (II) is bounded by $C\gamma^NJ(\bx_0)$. Last, combining the upper bound of (I), (II) and rearranging terms, we obtain \eqref{eq:relative-bound}.\\
\textbf{Part 2. $\delta$-Bellman inequality of $J_N(\bx)$:} We consider the Bellman type of inequality on $J_N(\bx_0),\,\bx_0\in\mathbb{X}$ under first optimal control input $u_0^F$, then,
{\small{\begin{align*}
&\gamma J_N(\bx_1^F)-J_N(\bx_0)
\\\overset{(iv)}{\le}&
\Bigg[
\sum_{t=1}^{N-1} \gamma^{t}l(\bx_t^F,\bu_t^F)
+
\gamma^{N}l(\bx_N^F,\bu_N^\star)
+
\gamma^{N+1}F(\bx_{N+1}^\dagger)
\Bigg]\\
&
-
\Bigg[
l(\bx_0,\bu_0^F)
+
\sum_{t=1}^{N-1} \gamma^tl(\bx_t^F,\bu_t^F)
+
\gamma^NF(\bx_{N}^F)
\Bigg]
\\
=&
-l(\bx_0,\bu_0^F)
+
\gamma^Nl(\bx_N^F, \bu_N^\star)
+
\gamma^{N+1}F(\bx_{N+1}^\dagger)
-
\gamma^NF(\bx_{N}^F)\\
\overset{(v)}{\leq}&
-l(\bx_0,\bu_0^F)
+
\gamma^Nl(\bx_N^F, \bu_N^\star)
+
\gamma^{N+1}J(\bx_{N+1}^\dagger)
-
\gamma^N F(\bx_{N}^F)\\
\overset{(vi)}{\leq}&
-l(\bx_0,\bu_0^F)
+
\gamma^NJ(\bx_N^F)\\
\overset{(vii)}{\leq}&
-l(\bx_0,\bu_0^F)
+C\gamma^N J(\bx_{0}),
\end{align*}}}
where $(iv)$ follows by applying the shifted control sequence 
$\{\bu_t^F\}_{t=1}^{N-1}$ obtained from \eqref{approx} at $\bx_0$ to the 
state $\bx_1^F$, and appending the infinite-horizon optimal input $\bu_N^\star$ to get $\bx_{N+1}^\dagger$. $(v)$ comes from the fact that $F(\bx_{N+1}^\dagger)\leq J(\bx_{N+1}^\dagger)$. $(vi)$ is obtained from throw $F(\bx_N^F)\geq0$ away and have $l(\bx_N^F, \bu_N^\star)+\gamma J(\bx_{N+1}^\dagger)=J(\bx_N^F)$ by Bellman equation. Finally, by applying $J(\bx_N^F)\leq CJ(\bx_0)$, we obtain $(vii)$. \\
Considering the $v$-relative lower boundness (Assumption \ref{a52}) and changing $J(\bx_0)$ by its upper bound $\frac{J_N(\bx_0)}{1-\gamma^N(1+C)}$, we have,
{\small{\begin{align*}
&\left(1
+\frac{C\gamma^N}{1-\gamma^N(1+C)} \right)J_N(\bx_{0})-l(\bx_0,\bu_0^F)\geq \gamma J_N(\bx_1^F)\\
&\left(1
+\frac{C\gamma^N}{1-\gamma^N(1+C)} \right)J_N(\bx_{0})-\gamma J_N(\bx_1^F)\geq l(\bx_0,\bu_0^F)\\
&\left(1-\frac{C\gamma^N}{v(1-\gamma^N(1+C))}\right)l(\bx_0,\bu_0^F)\leq J_N(\bx_{0})-\gamma J_N(\bx_1^F).
\end{align*}}}
Thus by assigning $\delta\leq1-\frac{C\gamma^N}{v(1-\gamma^N(1+C))}$, we finished the proof. It is remarkable that we need $N> \frac{\log\left({1+C}\right)}{\log\left(\frac{1}{\gamma}\right)}$ to ensure $1-\gamma^N(1+C)>0$ to let the $\delta$-bound effective.
\section{Proof of Theorem \ref{thm:performance-guarantees}}\label{ap4}
From $J(\bx)\geq J_N(\bx),\,\forall \bx\in\mathbb{X}$ (Lemma \ref{lem-2i}), we may bound the objective as
\[
\sup_{\bx\in\mathbb{X}}\frac{J_{\pi_\mathcal{D}}(\bx)-J(\bx)}{J(\bx)+\eta}
\leq 
\sup_{\bx\in\mathbb{X}}\frac{J_{\pi_\mathcal{D}}(\bx)-J_N(\bx)}{J_N(\bx)+\eta} .
\]
Hence, a sufficient condition on radius $r_{\bx}=\|\bx-\tilde\bx_i\|,\,\exists i\in[|\mathcal{D}|]$ is described as follows.
\begin{align*}
J_{\pi_\mathcal{D}}(\bx)-J_N(\bx)
&=\underbrace{J_{\pi_\mathcal{D}}(\bx)-\tilde{\mathbf{J}}_i}_{(\mathrm{III})}
+\underbrace{\tilde{\mathbf{J}}_i-J_N(\bx)}_{(\mathrm{IV})}\\
&\overset{(xiii)}{\leq}
(\delta^{-1}-1)\tilde{\mathbf{J}}_i+\delta^{-1}\lambda r_\bx+ L_J r_\bx\\
&{\color{red}{\overset{(xiv)}{\leq}}}
\mu(\tilde{\mathbf{J}}_i-L_Jr_\bx+\eta)
\overset{(xv)}{\leq}
\mu(J_N(\bx)+\eta),
\end{align*}  
The term $(\mathrm{III})$ follows from Theorem \ref{thm2}.  
The term $(\mathrm{IV})$ is upper bounded by the Lipschitz continuity of $J_N(\bx)$ (Assumption \ref{a54}), which yields $(xiii)$.  
{\color{red}{Inequality $(xiv)$ is the condition we want to enforce}}, and $(xv)$ again follows from the Lipschitz continuity in Assumption \ref{a54}. To guarantee {\color{red}{$(xiv)$}}, given $1-\delta^{-1}+\mu\geq 0$, the covering radius $r_\bx$ should satisfy,
\begin{align*}
(\delta^{-1}-1)\tilde{\mathbf{J}}_i+\delta^{-1}\lambda r_\bx+ L_J r_\bx
&\leq \mu(\tilde{\mathbf{J}}_i-L_Jr_\bx+\eta)\\
\delta^{-1}\lambda r_\bx+(1+\mu) L_J r_\bx
&\leq (1-\delta^{-1}+\mu)\tilde{\mathbf{J}}_i+\mu \eta\\
r_\bx
&\leq
\frac{(1-\delta^{-1})\tilde{\mathbf{J}}_i+\mu(\tilde{\mathbf{J}}_i+\eta)}{\delta^{-1}\lambda+(1+\mu) L_J},
\end{align*}

Therefore, by assigning $\tilde{\mathbf{J}}_i=0$ from $\inf_{\bx\in\mathbb{X}}J_N(\bx)\geq0$ (Assumption \ref{a501}), the sample complexity of Algorithm \ref{alg:1} is
\[
\mathcal{O}\!\left(
N_{\text{cover}}\!\left(
\frac{\mu \eta}{\delta^{-1}\lambda+(1+\mu)L_J},\,\mathbb{X}
\right)
\right).
\]
\end{document}